\documentclass[12pt]{amsart}
\usepackage{amssymb}
\usepackage{verbatim}
\usepackage{hyperref}
\usepackage{color}

\begin{document}

% define theorem environments
\newtheorem{theorem}{Theorem}    %[section]
\newtheorem{proposition}[theorem]{Proposition}
\newtheorem{conjecture}[theorem]{Conjecture}
\def\theconjecture{\unskip}
\newtheorem{corollary}[theorem]{Corollary}
\newtheorem{lemma}[theorem]{Lemma}
\newtheorem{sublemma}[theorem]{Sublemma}
\newtheorem{fact}[theorem]{Fact}
\newtheorem{observation}[theorem]{Observation}
\theoremstyle{definition}
\newtheorem{definition}[theorem]{Definition}
\newtheorem{notation}[theorem]{Notation}
\newtheorem{remark}[theorem]{Remark}
\newtheorem{question}[theorem]{Question}
\newtheorem{questions}[theorem]{Questions}
\newtheorem{example}[theorem]{Example}
\newtheorem{problem}[theorem]{Problem}
\newtheorem{exercise}[theorem]{Exercise}

\numberwithin{theorem}{section}
\numberwithin{equation}{section}

\def\reals{{\mathbb R}}
\def\torus{{\mathbb T}}
\def\integers{{\mathbb Z}}
\def\naturals{{\mathbb N}}
\def\complex{{\mathbb C}\/}
\def\distance{\operatorname{distance}\,}
\def\support{\operatorname{support}\,}
\def\dist{\operatorname{dist}\,}
\def\Span{\operatorname{span}\,}
\def\degree{\operatorname{degree}\,}
\def\dim{\operatorname{dim}\,}
\def\codim{\operatorname{codim}}
\def\trace{\operatorname{trace\,}}
\def\Span{\operatorname{span}\,}
\def\dimension{\operatorname{dimension}\,}
\def\codimension{\operatorname{codimension}\,}
\def\nullspace{\scriptk}
\def\kernel{\operatorname{Ker}}
\def\Re{\operatorname{Re\,} }
\def\Im{\operatorname{Im\,} }
\def\eps{\varepsilon}
\def\lt{L^2}
\def\diver{\operatorname{div}}
\def\curl{\operatorname{curl}}
\def\expect{\mathbb E}
\def\bull{$\bullet$\ }
\def\det{\operatorname{det}}
\def\Det{\operatorname{Det}}

\newcommand{\norm}[1]{ \|  #1 \|}
\newcommand{\Norm}[1]{ \Big\|  #1 \Big\| }
\newcommand{\set}[1]{ \left\{ #1 \right\} }
\def\one{{\mathbf 1}}
\newcommand{\modulo}[2]{[#1]_{#2}}
\newcommand{\abr}[1]{ \langle  #1 \rangle}

\renewcommand{\qed}{\hfill \mbox{\raggedright \rule{0.1in}{0.1in}}}

\def\bp{\mathbf p}
\def\bff{\mathbf f}
\def\bg{\mathbf g}

\def\b1{\mathbf 1}
\def\bi{\mathbf i}
\def\bj{\mathbf j}
\def\bk{\mathbf k}

\def\bu{\mathbf u}
\def\bv{\mathbf v}
\def\bw{\mathbf w}
\def\bx{\mathbf x}
\def\by{\mathbf y}

\def\scriptf{{\mathcal F}}
\def\scriptg{{\mathcal G}}
\def\scriptm{{\mathcal M}}
\def\scriptb{{\mathcal B}}
\def\scriptc{{\mathcal C}}
\def\scriptt{{\mathcal T}}
\def\scripti{{\mathcal I}}
\def\scripte{{\mathcal E}}
\def\scriptv{{\mathcal V}}
\def\scriptw{{\mathcal W}}
\def\scriptu{{\mathcal U}}
\def\scriptS{{\mathcal S}}
\def\scripta{{\mathcal A}}
\def\scriptr{{\mathcal R}}
\def\scripto{{\mathcal O}}
\def\scripth{{\mathcal H}}
\def\scriptd{{\mathcal D}}
\def\scriptl{{\mathcal L}}
\def\scriptn{{\mathcal N}}
\def\scriptp{{\mathcal P}}
\def\scriptk{{\mathcal K}}
\def\scriptP{{\mathcal P}}
\def\scriptj{{\mathcal J}}
\def\frakv{{\mathfrak V}}
\def\frakG{{\mathfrak G}}
\def\frakA{{\mathfrak A}}

\author{Marcos Charalambides}
\address{
        Marcos Charalambides\\
        Department of Mathematics\\
        University of California \\
        Berkeley, CA 94720-3840, USA}
\email{marcos@math.berkeley.edu}

%\author{Michael Christ}
%\address{
%       Michael Christ\\
%        Department of Mathematics\\
%        University of California \\
%        Berkeley, CA 94720-3840, USA}
%\email{mchrist@math.berkeley.edu}
\thanks{Research supported in part by NSF grant DMS-0901569.}

\date{\today}

\title
[Distinct distance subsets]
{A note on distinct distance subsets}

\maketitle
\begin{abstract}
It is shown that given a set of $N$ points in the plane or on the sphere, there is a subset of size $\gtrsim N^{1/3}/\log N$ with all pairwise distances between points distinct.
\end{abstract}

\section{Introduction}

Given a finite set $P$ of points in $\reals^2$, we may ask how large a subset $Q\subset P$ can be if all the ${|Q|}\choose{2}$ pairwise distances between elements of $Q$ are distinct. Define $\Delta(P)$ to be the maximal cardinality of such a subset $Q$.

A problem in the family of problems related to the Erd\H{o}s distinct distance problem \cite{erdos} first posed by Avis, Erd\H{o}s and Pach \cite{avep91}, asks how small the quantity $\Delta(P)$ can be as a function of the number of points of $P$. For a positive integer $N$, write $\delta(N)$ for the minimum of $\Delta(P)$ over all $N$-element sets $P\subset\reals^2$.

\begin{question}
What is the order of magnitude of $\delta(N)$?
\end{question}

See \cite{book} for a survey of this and several related problems. 

Since a $\sqrt{N}\times\sqrt{N}$ integer grid determines $\lesssim N/\sqrt{\log N}$ distinct distances, one obtains the following upper bound for $\delta(N)$.

\begin{proposition}\label{propy}
$\delta(N)\lesssim N^{1/2}(\log N)^{-1/4}$
\end{proposition}

\section{An improved lower bound}

Using the probabilistic method combined with a deletion argument, Lefmann and Thiele \cite{leffthiele} obtained the lower bound $\gtrsim N^{1/4}$. Dumitrescu \cite{dumi} improved this lower bound to $\gtrsim N^{0.288}$, by using an upper bound of $\lesssim N^{2.136}$ for the number of isosceles triangles determined by $N$ points in the plane due to Pach and Tardos \cite{pachtardos}. 

The purpose of this note is to observe that the aforementioned method of Lefmann and Thiele combined with the Guth-Katz bound \cite{guthkatz} on the number of quadruples $(p_1,p_2,q_1,q_2)\in P^4$ such that $\norm{p_1-p_2}=\norm{q_1-q_2}$ gives an improved lower bound for $\delta(N)$.

\begin{proposition}\label{proplow}
$\delta(N)\gtrsim N^{1/3}/\log N$
\end{proposition}

\begin{proof}
Let $P\subset\reals^2$ have cardinality $N>2$. Let $q\in[0,1]$ be a probability which will be chosen below. Choose a random subset $Q$ of $P$; each element of $P$ is selected independently with probability $q$ and rejected with probability $(1-q)$. 

For a given subset $X$ of $\reals^2$, write $t(X)$ for the number of isosceles triangles in $X$, that is the number of ordered triples $(p,q_1,q_2)$ of \emph{distinct} elements of $X$ such that $\norm{p-q_1}=\norm{p-q_2}$. Write $f(X)$ for the number of quadruples $(p_1,p_2,q_1,q_2)$ of \emph{distinct} elements of $X$ such that $\norm{p_1-p_2}=\norm{q_1-q_2}$.

There are two ways that $Q$ can fail to have all pairwise distances distinct: precisely when either $t(Q)$ or $f(Q)$ are non-zero. For each isosceles triangle $(p,q_1,q_2)$ contributing to $t(Q)$, we delete one of the three points $p,q_1,q_2$ from $Q$ to obtain a subset $Q'\subset Q$ such that $t(Q')=0$ and $|Q'|\ge|Q|-t(Q)$. Similarly, we delete a point from each quadruple contributing to $f(Q)$ to obtain a subset $Q''\subset Q$ such that $t(Q'')=f(Q'')=0$ and \[|Q''|\ge |Q|-t(Q)-f(Q).\] Taking expectations we obtain \begin{equation}\label{main}\mathbb{E}(|Q''|)\ge \mathbb{E}(|Q|)-\mathbb{E}(t(Q))-\mathbb{E}(f(Q)).\end{equation}

By construction, $\mathbb{E}(|Q|)=qN$. When choosing $Q$ from $P$, each triple of distinct points of $P$ survives with probability $q^3$ and each quadruple survives with probability $q^4$. Thus, $\mathbb{E}(t(Q))\le q^3t(P)$ and $\mathbb{E}(f(Q))\le q^4f(P)$. 

Pach and Sharir \cite{pachshar} observed that the Szemer\'edi-Trotter theorem \cite{szemtrot} readily implies that $t(P)\lesssim N^{7/3}$. Pach and Tardos \cite{pachtardos} proved the stronger bound $t(P)\lesssim N^{2.136}$, but we will not need this improvement in the argument.

The Guth-Katz theorem \cite{guthkatz} implies that $f(P)\lesssim N^3\log N$. 

Plugging these bounds into (\ref{main}) gives \[\mathbb{E}(|Q''|)\ge qN - q^3aN^{7/3}-q^4bN^3\log N\] for some positive (universal) constants $a, b$. Setting $q=N^{-2/3}(\log N)^{-1}$ gives the bound \[\mathbb{E}(|Q''|)\ge \frac{N^{1/3}}{\log N}(1- a(\log N)^{-3} - b(\log N)^{-3}).\] 

By the pigeonhole principle, there thus exists a subset of $P$ of size $\gtrsim N^{1/3}/\log N$ with all pairwise distances distinct.
\end{proof}

\begin{remark}
By scaling the probability $q$ appropriately, the same proof shows that given any fixed $K>0$ there exists a positive integer $N_K$ such that \[\delta(N)\ge (K-o_K(1))\frac{N^{1/3}}{\log N}\] whenever $N>N_K$.
\end{remark}

The Guth-Katz bound is optimal up to constants \cite{guthkatz}, so it seems that the $1/3$ in the exponent is the best one can achieve with the method of Lefmann and Thiele (in the form used above). Nonetheless, the argument is rather wasteful and this author is inclined to conjecture that the order of magnitude of $\delta(N)$ should be closer to the upper bound of Proposition~\ref{propy}.

\begin{conjecture} Given $\epsilon>0$, there exists some constant $c_\epsilon>0$ such that \[\delta(N)\ge c_\epsilon N^{1/2-\epsilon}.\]
\end{conjecture}

\section{Distinct distance subsets on the sphere}

We may also ask the analogous question for finite subsets $P$ of the two-dimensional sphere $S$ and the analogous function $\delta_S(N)$. In this case, $N$ equally spaced points on a great circle determine $\lesssim N$ distinct distances, so one has the upper bound \[\delta_S(N)\lesssim \sqrt{N}.\] In this situation we can, in fact, obtain the same lower bound as for the plane.

\begin{proposition}\label{propsph}
$\delta_S(N)\gtrsim N^{1/3}/\log N.$
\end{proposition}

The Szemer\`edi-Trotter theorem \cite{szemtrot} can be applied to this situation as well in a very similar way to how it was used in \cite{pachshar} to obtain the same upper bound on $t_S(P)$ (that is, the number of spherical isosceles triangles determined by $P$).

\begin{lemma}
$t_S(P)\lesssim N^{7/3}.$
\end{lemma}

\begin{proof}
If $(p,q_1,q_2)$ is an isosceles triangle with $\norm{p-q_1}=\norm{p-q_2}$ then say that $q_1$ and $q_2$ are the \emph{base vertices} of the triangle. It suffices to show that the number of isosceles triangles with a fixed $q\in P$ as a base vertex is $\lesssim N^{4/3}$.

Let $q\in P$. For each point $p\in P\setminus\{q\}$, consider the spherical circle $\Gamma_p\subset S$ with center $p$ which passes through $q$. Any point on $\Gamma_p\cap P$ other than $q$ gives rise to an isosceles triangle with $q$ as a base vertex and, moreover, all isosceles triangles with $q$ as a base vertex arise in this way. Thus, to count the number of isosceles triangles with $q$ as a base vertex it is enough to count the number of incidences between the $N-1$ points $P\setminus\{q\}$ and the $N-1$ circles $\{\Gamma_p\,\,|\,\, p\in P\setminus\{q\}\}$. By embedding $S$ as the unit sphere in $\reals^3$ with $q$ as the north pole and applying a stereographic projection to the plane through the equator, all the circles become lines (since each one passes through $q$). By the Szemer\'edi-Trotter theorem, the number of incidences is bounded above by $\lesssim N^{4/3}$.
\end{proof}

It is known that the Guth-Katz bound also applies to the sphere (see, for example, the entry \cite{taoblog} on Terence Tao's internet blog); this gives $f_S(P)\lesssim N^3\log N$.

With these bounds on $t_S(P)$ and $f_S(P)$, the argument in the proof of Proposition~\ref{proplow} applies equally well to the situation of points on a sphere. This completes the proof of Proposition~\ref{propsph}.
\qed
\\

\paragraph{\bf{Acknowledgement}}

The author thanks M. Christ and A. Dumitrescu for suggestions on improving a preliminary version of this note.

\bibliographystyle{amsplain}{}
\bibliography{distbib}

\end{document}